\documentclass[11pt]{amsart}
\IfFileExists{srcltx.sty}{\usepackage[active]{srcltx}}{}

\usepackage{amscd}
\usepackage{amsfonts,amssymb,latexsym}
\newcommand{\SC}{{\mathcal{C}}}

\newcommand{\SE}{{\mathcal{E}}}
\newcommand{\SF}{{\mathcal{F}}}

\newcommand{\SM}{{\mathcal{M}}}

\newcommand{\SO}{{\mathcal{O}}}

\newcommand{\ox}{\otimes}

\newcommand{\PP}{\mathbb{P}}
\newcommand{\ZZ}{\mathbb{Z}}

\newcommand{\CC}{\mathbb{C}}

\newcommand{\QQ}{\mathbb{Q}}
\newcommand{\VV}{\mathbb{V}}

\newcommand{\Ext}{\operatorname{Ext}}

\newcommand{\Spec}{\operatorname{Spec}}

\newcommand{\codim}{\operatorname{codim}}

\newcommand{\Hom}{\operatorname{Hom}}

\newcommand{\Pic}{\operatorname{Pic}}

\newcommand{\surj}{\twoheadrightarrow}
\newcommand{\inj}{\hookrightarrow}
\newcommand{\too}{\longrightarrow}
\newcommand{\rk}{\operatorname{rk}}
\newcommand{\End}{\operatorname{End}}
\newcommand{\wt}{\widetilde}
\newcommand{\tr}{\operatorname{tr}}
\newcommand{\GL}{\operatorname{GL}}

\newcommand{\PGL}{\operatorname{PGL}}

\newtheorem{proposition}{Proposition}[section]
\newtheorem{theorem}[proposition]{Theorem}

\newtheorem{lemma}[proposition]{Lemma}

\newtheorem{remark}[proposition]{Remark}

\numberwithin{equation}{section}

\begin{document}

\title[Automorphisms of moduli spaces]{Automorphisms of moduli
spaces of vector bundles over a curve}

\author[I. Biswas]{Indranil Biswas}
 \address{School of Mathematics, Tata Institute of Fundamental
 Research, Homi Bhabha Road, Bombay 400005, India}
\email{indranil@math.tifr.res.in}

\author[T. L. G\'omez]{Tom\'as L. G\'omez}
\address{Instituto de Ciencias Matem\'aticas (CSIC-UAM-UC3M-UCM),
Serrano 113bis, 28006 Madrid, Spain; and
Facultad de Ciencias Matem\'aticas,
Universidad Complutense de Madrid, 28040 Madrid, Spain. }
\email{tomas.gomez@icmat.es}

\author[V. Mu\~{n}oz]{Vicente Mu\~{n}oz}
\address{Facultad de Ciencias Matem\'aticas,
 Universidad Complutense de Madrid, 28040 Madrid, Spain}
\email{vicente.munoz@mat.ucm.es}

\subjclass[2000]{14H60}

\keywords{Stable bundles, moduli space, automorphism group}

\date{}

\thanks{Partially supported by grant MTM2007-63582 of
the Ministerio de Ciencia e Innovaci\'on (Spain).}

\begin{abstract}
Let $X$ be an irreducible smooth complex projective curve
of genus $g\geq 4$.
Let $M(r,\Lambda)$ be the moduli space of stable vector bundles $E\too
X$ or rank $r$ and fixed determinant $\Lambda$.
We show that the automorphism group of $M(r,\Lambda)$ is generated by
automorphisms of the curve $X$, tensorization with suitable line bundles,
and, if $r$ divides $2\text{deg}(\Lambda)$, also dualization of
vector bundles.
\end{abstract}

\maketitle

\section{Introduction}

Let $X$ be a smooth complex projective curve of genus $g$,
with $g\, \geq \, 4$. Fix a line bundle $\Lambda$ on $X$ of degree $d$.
Let $M(r,\Lambda)$ be the moduli space of stable vector bundles $E\too X$
or rank $r$ and determinant $\det (E) \cong \Lambda$.

There are two obvious ways of producing automorphisms of $M(r,\Lambda)$ :
\begin{enumerate}
\item Let $\sigma:X\too X$ be an automorphism, and let $L$ be a
line bundle with $L^r\otimes \sigma^*\Lambda\cong \Lambda$.
Send $E$ to $L\otimes \sigma^* E$.
\item Let $\sigma:X\too X$ be an automorphism, and let $L$ be a
line bundle with $ L^r\otimes\sigma^*\Lambda^{-1}\cong \Lambda$.
Send $E$ to $L\otimes \sigma^* (E^\vee)$.
\end{enumerate}
Note that the second type can only occur if $r|2d$.
Also note that both operations send stable bundles to stable
bundles, so they do define automorphisms of the moduli space
$M(r,\Lambda)$.

We aim is to give a new proof of the following result of
\cite{KP} and \cite{HR}.

\begin{theorem} \label{main}
The automorphism group of $M=M(r,\Lambda)$
is given by the automorphisms described above.
\end{theorem}

Theorem \ref{main} was initially proved by Kouvidakis and Pantev \cite{KP}.
The proof of \cite{KP} uses the Hitchin map defined on the moduli of Higgs bundles,
giving a delicate argument in which one has to use the geometry of the fibers of the
Hitchin map over spectral curves of three types: smooth curves, singular curves
with a single node, and some singular curves with
particular type of singularities.

In \cite{HR}, Hwang and Ramanan gave different proof of Theorem \ref{main}.
The proof of \cite{HR} is simpler in spirit (although 
it is a bit technical for the case $g=4$). First, they determine
geometrically the Hitchin discriminant (the locus of singular
spectral curves), and then they go to the dual variety of this Hitchin
discriminant, which they prove to be isomorphic to a
locus of Hecke transforms. For this, one needs to use the theory of minimal rational
curves in the moduli space.

We show that the proof of \cite{HR} can be
largely simplified. We prove Theorem \ref{main} by reconstructing, from
the Hitchin discriminant, the nilpotent cone bundles, which determine the
automorphisms of a generic bundle directly. This avoids passing through
the
dual variety; it also avoids using both the Hecke transforms, and the
minimal rational curves. Moreover, we cover all cases $g \geq 4$ with 
no special arguments for low genus.

It is worth pointing out that the argument provided here can be generalized
to other moduli spaces, like the moduli space of symplectic bundles
\cite{us}.

As a byproduct of our analysis, we also obtain the following
Torelli type theorem for the moduli space $M(r,\Lambda)$. This is
already well-known \cite{MN, NR, NR2, KP, HR}.

\begin{theorem} \label{torelli}
 If $M_X(r,\Lambda) \cong M_{X'}(r',\Lambda')$ then $X\cong X'$ and $r=r'$.
\end{theorem}

\section{Moduli space of bundles}

Let $X$ be a smooth complex projective curve of genus $g$,
with $g\, \geq \, 3$. Fix $r\geq 2$, and also fix a line bundle
$\Lambda$ on $X$ of degree $d$.
A vector bundle $E\too X$ is called \textit{stable}
(respectively, \textit{semistable}) if, for
all proper subbundles $E'\subset E$ of positive rank,
 $$
 \frac{\deg E'}{\rk E'} < \frac{\deg E}{\rk E} \quad
 \text{(respectively, $\frac{\deg E'}{\rk E'}\leq \frac{\deg E}{\rk
E}$)}\, .
 $$

Let $M(r,\Lambda)$ be the moduli space of stable bundles $E\too X$
or rank $r$ and determinant $\det (E) \cong \Lambda$.

The following results are needed later. They appear in  \cite{HR} but we
give a simpler proof.

\begin{lemma} \label{lem:asdfg1}
Assume that $d \,\leq\, r(g-1)$ and $r\geq 2$.

Then $H^0(E)\, =\, 0$ for a general stable bundle $E\,\in\,
M(r,\Lambda)$.

Moreover, if ${\mathcal F}  \subset \text{Jac}^0 X$ is a $\kappa$-dimensional family
of line bundles of degree zero and $d\, \leq \, r(g-1)-\kappa$,
then for generic $E$,
$$H^0(E\otimes L)\, =\, 0$$ for all $L\in  \SF$.
\end{lemma}

\begin{proof}
Recall that $\dim M(r,\Lambda) = (r^2-1)(g-1)$.
We want to show that the codimension of the subset
 $$
 \big\{ E \,\in\, M(r,\Lambda)\, \mid\, H^0(E)\, \not=\, 0\big\}
 \, \subset\, M(r,\Lambda)
 $$
is positive. Suppose that $E \,\in\, M(r,\Lambda)$
and  $s \in H^0(E)\setminus \{0\}$. Then there is a
short exact sequence
 \begin{equation}\label{eqn:ext}
 0\too L \stackrel{s}{\too} E \too Q \too 0\, ,
 \end{equation}
where $L\,=\, {\mathcal O}_X(D)$ with
$D$ being the effective divisor defined by $s$. Write $\ell = \deg D\geq
0$. Let us compute the dimension of the space of extensions
as in (\ref{eqn:ext}).

The line bundles $L$ are parametrized by a family of dimension at most
$\ell$. The family of vector bundles $Q$ is bounded, and the
dimension of the parameter space is at most $((r-1)^2-1)(g-1)$.
Finally, $\Hom(Q,L)=0$, since otherwise there would be a non-constant
non-zero endomorphism of $E$ given by the composition
$E\twoheadrightarrow Q \to L \hookrightarrow E$, which is not
allowed by the stability condition of $E$. So
  $$
   \dim \Ext^1(Q,L) = -\chi(L\ox Q^*)= - \deg(L \ox Q^*)+ (r-1)(g-1) = d-r \ell +(r-1)(g-1)\, ,
  $$
since $\deg Q=d-\ell$. So the dimension of the family of $E$ in (\ref{eqn:ext}) is
  \begin{eqnarray*}
  &\leq& \ell + ( (r-1)^2-1)(g-1) + d-r \ell +(r-1)(g-1)-1  \\
  &<& (r^2-1)(g-1) -(r(g-1)-d)\, .
  \end{eqnarray*}
So if $d \leq r(g-1)$, then the family of $E$ as in (\ref{eqn:ext}) is
of
positive codimension.

For the second part of the lemma, just note that the vector bundles $E$
such that
$H^0(E\otimes L')\neq 0$ for some $L'\in \SF$ lie in a short exact
sequence like
(\ref{eqn:ext}) with $L \,=\, L'\otimes {\mathcal O}_X(D)$. The space of
such bundles $L$ is of dimension $\leq \ell + \kappa$. The result
follows analogously as before.
\end{proof}

For a bundle $E$, let $\End_0 E$ denotes the
sheaf of endomorphisms of trace zero. For any divisor
$D$, define $\End_0 E(D)\, :=\, (\End_0 E)\otimes
{\mathcal O}_X(D)$.

\begin{lemma}\label{lem:end-zero}
Suppose that $g\geq 2\ell +2$.
Then $H^0(\End_0 E (D))=0$ for a general stable bundle $E \in
M(r,\Lambda)$ and any effective divisor $D$ of degree $\ell$.
\end{lemma}

\begin{proof}
We will prove this by induction on $r$. The result is obvious for $r=1$,
so assume $r\geq 2$.
Tensoring with a line bundle $\mu$ produces an isomorphism
 $$M(r,\Lambda)\, \stackrel{\sim}{\longrightarrow}\, M(r,\Lambda\otimes \mu^{\otimes r}),
~\,~\, E\,\longmapsto\, E\otimes\mu\, .
 $$
Dualizing produces an isomorphism 
 $$M(r,\Lambda)
\, \stackrel{\sim}{\longrightarrow}\,
M(r,\Lambda^*), ~\,~\, E\,\longmapsto\, E^*\, .
 $$
In view of these isomorphisms, it suffices to prove the lemma for particular
cases. So we can arrange that $0\leq d \leq r/2$.

Using induction, a generic vector bundle $F\in M(r-1,\Lambda)$ satisfies the
condition that $H^0(\End_0 F(D))=0$,
for any $D\in X^{(\ell)}$, where $X^{(\ell)}\,=\, \text{Sym}^\ell(X)$ is the
set of effective divisors on $X$ of degree $\ell$. On the other hand, $H^0(F(D))=0$
and $H^0(F^*(D))=0$, for any $D\in X^{(\ell)}$. 
This is a consequence of Lemma \ref{lem:asdfg1}, because
  $$
  \pm d +(r-1)\ell \leq \frac{r}2 + (r-1)\ell \leq (r-1)(g-1) - \ell \, ,
  $$
which is equivalent to $(r-2)(\ell + \frac12) + (r-1)(g-2\ell-2) \geq 0$.

Now let us construct a rank $r$ vector bundle as follows.
Take an extension 
 \begin{equation} \label{eqn:la-extension}
 0\too \SO \too E \too F\too 0\, .
  \end{equation}
Let $f\in H^0(\End E(D))$, i.e.,
$f:E\too E(D)$. Composing $f$ with
the inclusion $\SO\too E$ and the projection
$E(D)\too F(D)$, we get a section
of $F(D)$, which vanishes by construction. Therefore, $f$ induces
a homomorphism $f_2: F\too F(D)$.
By induction hypothesis, $f_2= s \, \operatorname{Id}$, where $s\in H^0(\SO(D))$.
The map $f':=f-s\, \operatorname{Id} \in H^0(\End E(D))$ has $f_2'=0$, where
$f'_2$ is constructed as above from $f'$.
Consequently, it induces
a map $f':E\too \SO(D)$, i.e., a section of $E^*(D)$.

Let us see that we can choose an extension (\ref{eqn:la-extension}) so that $H^0(E^*(D))=0$.
We have an exact sequence:
 $$
 0=H^0(F^*(D)) \to H^0(E^*(D)) \to H^0(\SO(D)) \to H^1(F^*(D)) \to \ldots
 $$
So we need an extension class $\xi \in H^1(F^*)$ so that the map 
$H^0(\SO(D)) \stackrel{\xi}{\too} H^1(F^*(D))$ is injective for all $D\in X^{(\ell)}$.
That is, $\xi$ does not map to zero under $H^1(F^*) \to H^1(F^*(D))$. Looking
at the exact sequence $0\to V_D:=F^*(D)|_D \to H^1(F^*) \to H^1(F^*(D))\to 0$, we
require $\xi \notin \bigcup_{D\in X^{(\ell)}} V_D$. This space is of dimension $\ell+
(r-1)\ell= r\ell$, and $\dim H^1(F^*)=d+(r-1)(g-1)$. So the condition is
 $$
 r\ell < d+(r-1)(g-1)\, .
 $$
As $d\geq 0$, we need $g\geq \frac{r}{r-1} (\ell+1)$. The minimum happens for $r=2$, and
it is $g\geq 2\ell +2$.

Therefore, $f'\in H^0(E^*(D))=0$ and $f=s \, \mathrm{Id}$. So 
we have constructed a vector bundle $E$ such that $H^0(\End_0 E(D))=0$ for any $D\in X^{(\ell)}$.
Now take a family $E_t$, $t\in T$, of vector bundles parametrized by an (open) curve $T$, such
that $E_0=E$ and $E_t$ is stable for generic $t$ (this is possible because the moduli stack of
vector bundles of fixed rank and determinant is irreducible and the stable locus
is open in it \cite{Hoff}). As the condition that
$H^0(\End_0 E(D))=0$ is open, and $X^{(\ell)}$ is compact, we conclude
that $H^0(\End_0 E_t (D))=0$ for generic $t$. Hence
$H^0(\End_0 \tilde{E} (D))=0$ for generic stable bundle $\widetilde{E}\in M(r,\Lambda)$.
\end{proof}

\begin{remark} \label{rem-v}
{\rm Since $g\, \geq\, 4$, the condition on $g$ in
Lemma \ref{lem:end-zero} is satisfied if $\ell \,=\, 1$.}
\end{remark}

\section{Hitchin discriminant}

Let us recall the definition of the Hitchin map (see \cite[Section
5.10]{Hi}). A Higgs bundle is a pair $(E,\theta)$, where
$\theta:E\too E\otimes K_X$. Denote $M=M(r,\Lambda)$ and
let ${\mathcal M}={\mathcal M}(r,\Lambda)$ be the moduli space
of semistable Higgs bundles of rank $r$ and determinant $\Lambda$.
The cotangent bundle $T^* M \subset {\mathcal M}$ is
an open subset.

The \textit{Hitchin map} 
 $$
 H:{\mathcal M} \too W:= H^0(X, K_X^2) \oplus \ldots \oplus H^0(X, K_X^{r}) \, ,
 $$
is defined by $H(E,\theta)=(s_2(\theta),\ldots, s_{r}(\theta))$, where
$s_i(\theta)\,:=\,{\rm tr}(\wedge^i \theta)$. This
restricts to $$h: T^*M \too W ~\,\text{~and~}\,~
 h_E: T^*_E M= H^0(X, \End_0 E\otimes K_X) \too W\, .
 $$

For an element $s=(s_2,\ldots, s_{r}) \in W$, the \textit{spectral curve} $X_s$ associated to
it is the curve in the total space $\VV(K_X)$ of $K_X$ is defined by the equation
\begin{equation}\label{eqn:spectral-curve}
y^{r} + s_2(x) y^{r-1} + \ldots + s_{r}(x)\,=\,0\, ,
\end{equation}
where $x$ is a coordinate for $X$, and $y$ is the tautological coordinate
$dx$ along the fibers of the projection $\VV(K_X)\too X$.

Consider
 $$
 {\mathcal D}\subset W
 $$
the divisor consisting of characteristic polynomials with singular
spectral curves.

\begin{proposition} \label{prop:W-and-D}
The following statements hold:
\begin{enumerate}
 \item For $w\in W-{\mathcal D}$, the fiber $h^{-1}(w)$ is an open subset of an abelian variety.
 \item For generic $w\in {\mathcal D}$, the fiber $h^{-1}(w)$ is an
  open subset of the uniruled variety. The complement of this open subset is of
codimension at least two.
\end{enumerate}
\end{proposition}

\begin{proof}
The map $H:{\mathcal M} \too W$ is proper. The inverse image
$H^{-1}(w)$ is an abelian variety for $w\in W - {\mathcal
D}$, by \cite{Hi}.

Now take $w\in W-{\mathcal D}$. If $w$ is generic, then $X_w$ has
a unique singularity which is a node.
Let $Y$ be an integral curve whose only singularity is one
simple node at a point $y$. Let $\pi_Y:\wt Y\too Y$ be the
normalization,
and let $x$ and $z$ be the preimages of $y$. The compactified
Jacobian $\overline{J}(Y)$, parametrizing torsionfree sheaves of
rank 1 and degree 0 on $Y$, is birational to a $\PP^1$-fibration
$P$ over $J(\wt Y)$, whose fiber over $L\in J(\wt Y)$
is $\PP(L_x\oplus L_z)$. The morphism $P\too \overline{J}(Y)$
is constructed as follows:

A point of $P$ corresponds to a
line bundle $L$ on $\wt Y$ and a one dimensional quotient
$q:L_x\oplus L_z \surj \CC$ (up to a scalar multiple).
This is sent to the torsionfree sheaf
$L'$ on $Y$ that fits in the short exact sequence
$$
0 \too L' \too (\pi_Y)_*L \stackrel{q}\too \CC_y \too 0\, .
$$
For the proof, see \cite[Theorem 4]{Bh}.

The complement
 $$
 {\mathcal M} - T^*M
 $$
is of codimension at least three (in \cite[Theorem II.6 (iii)]{Fa} it is proved that the complement
has codimension at least two under a weaker assumption, but if
we assume $g\geq 3$, then the same proof gives that the codimension
is at least three). Therefore, $({\mathcal M}- T^*{M})\cap {\mathcal D}$
is of codimension at least $2$
in ${\mathcal D}$,
so for generic $w\in {\mathcal D}$,
$$H^{-1}(w)-h^{-1}(w)\,\subset\, H^{-1}(w)$$ is of codimension at least $2$.
\end{proof}

\begin{proposition} \label{prop:irreducibility}
 The hypersurface $h^{-1}({\mathcal D})$ is irreducible.
\end{proposition}

\begin{proof}
Let ${\mathcal D}^* \subset {\mathcal D}$ be the complement of the subset
consisting of all spectral curves with a single node. Then ${\mathcal D}^*$ is
closed so ${\mathcal D}^o={\mathcal D} -{\mathcal D}^*$ is open.
The fibers of $h$ over points of ${\mathcal D}^o$ are irreducible by Proposition
\ref{prop:W-and-D}. Hence  $h^{-1}({\mathcal D}^o)$ is irreducible.
By Theorem II.5 of \cite{Fa}, the
fibers of the Hitchin map $H:{\mathcal M}\too W$ are Lagrangian
(hence their dimension is half the dimension of
$\mathcal H$). So the fibers of $H$ are equidimensional, and in particular,
the codimension of $h^{-1}({\mathcal D}^*)$ coincides
with that of ${\mathcal D}^*\subset
W$, which is at least two. Therefore $h^{-1}({\mathcal D})$ is an irreducible
divisor of $T^*M$.
\end{proof}

The inverse image $h^{-1}({\mathcal D})$ is called the \textit{Hitchin discriminant}.

\begin{theorem} \label{thm:Hitchin-discriminant}
 The Hitchin discriminant $h^{-1}({\mathcal D})$ is the closure of the union of
 the (complete) rational curves in $T^* M$.
\end{theorem}

\begin{proof}
 Let $l\cong \PP^1\subset h^{-1}({\mathcal D})$. Then $h(l)\subset W$. As it is
a complete curve in an affine variety, it is a point. So $l$ is included in a fiber. By Proposition
\ref{prop:W-and-D}, it cannot be contained in a fiber over $w\in W-{\mathcal D}$.

Now let $w\in {\mathcal D}^o$. Then Proposition
\ref{prop:W-and-D} again shows that there is a family of $\PP^1$ generically covering the
fiber over $w$. Now using Proposition \ref{prop:irreducibility}, we get that
the closure is the entire $h^{-1}({\mathcal D})$.
\end{proof}

\section{Torelli theorem}

This section is devoted to a \emph{Torelli type theorem} for the moduli space $M=M(r,\Lambda)$.

\begin{lemma} \label{lem:4.1}
The global algebraic functions $\Gamma(T^*M)$ produce a map
 $$
 \widetilde{h}: T^*M \too \Spec(\Gamma(T^*M))\cong W
 \cong \CC^N
 $$
which is the Hitchin map up to an automorphism of $\CC^N$, where
$N=\dim M$.

Moreover, consider the standard dilation action of $\CC^*$ on the fibers
of $T^*M$. Then
there is a unique $\CC^*$-action ``$\cdot$'' on $W$ such that $\widetilde{h}$
is $\CC^*$-equivariant, meaning $\widetilde{h}(E,\lambda\theta)= \lambda
\cdot \widetilde{h}(E,\theta)$.
\end{lemma}

\begin{proof}
This holds for the Hitchin map $H$ on the moduli
of semistable Higgs bundles $\SM$
(cf. \cite{Hi}). On the other hand, the general fiber
of $H$ is smooth, and the codimension
of $T^*M \subset \SM$ on these fibers
is at least two
(cf. \cite[Theorem II.6 (i)]{Fa}, and note that $T^*M$
is a subset of the moduli $\SM^0 \subset \SM$ of stable Higgs
bundles). Therefore, e conclude that
the lemma also holds
for the restriction of the Hitchin map to the
cotangent bundle $T^*M$.

The last assertion is clear.
\end{proof}

We note that the $\CC^*$-action on $W$ allows to recover the space
 $$
 W_r \,:=\, H^0( K^{r}_X)\,\subset\, W
 $$
uniquely as the subset where the rate of decay is $|\lambda|^r$, which is the maximum possible.

\begin{proposition}\label{prop:4.2}
The intersection $\SC:={\mathcal D}\cap W_r\subset W_{r}=H^0(K_X^{r}) \subset W$ is irreducible.
Moreover, ${\mathbb P}({\mathcal C})\subset {\mathbb P}(W_{r})$
 is the dual variety of $X\subset {\mathbb P}(W_{r}^*)$ for the embedding
 given by the linear series $|K_X^{r}|$.
\end{proposition}

\begin{proof}
A spectral curve corresponding to a nonzero section
$s_{r}\in W_r=H^0(K_X^{r})$
has equation $y^{r}+s_{r}(x)=0$, and this curve
is singular at the points with coordinates $(x,0)$
such that $x$ is a zero of $s_{r}$ of order at least
two.
We have $s_r\in {\mathcal C}$ if and only if there is some $x_0\in X$
such that $s_r$ vanishes at $x_0$ of order at least two. Therefore
$s_r\in H^0(K_X^{r}(-2x_0)) \subset H^0(K_X^{r})$. From this
the second statement follows, taking into
account that the linear system
$\vert K_X^{r}\vert $ is very ample, so $X$ is embedded.
\end{proof}

Denote
 $$
 {\mathcal C}_x= H^0(K_X^{r}(-2x)) \subset W_{r} \, .
 $$
Then ${\mathcal C}= \bigcup_{x\in X} {\mathcal C}_x$.
There is a unique rational map ${\mathcal C}\too X$ which sends ${\mathcal C}_x$ to
$x$. This is uniquely determined, up to an automorphism of $X$, by the
property that it is the only rational map with connected rational
fibers, up to an automorphism of $X$.

\begin{theorem}[Torelli] \label{thm:torelli}
 Let $X$ and $X'$ be two smooth projective curves of genus $g\geq 3$, and consider
 two moduli spaces of stable vector bundles $M=M_X(r,\Lambda)$ and $M'=M_{X'}(r',\Lambda')$
 on $X$ and $X'$ respectively.
 If these moduli spaces are isomorphic, then $X\cong X'$.
\end{theorem}

\begin{proof}
Suppose $\Phi: M\too M'$ is an isomorphism. Then
there is an isomorphism $d\Phi: T^*M \too T^*M'$. By Lemma \ref{lem:4.1}, there
is a commutative diagram
 $$
 \begin{array}{ccc}
 T^*M  & \stackrel{d\Phi}{\too} & T^*M' \\
 \Big\downarrow & & \Big\downarrow \\
 W & \stackrel{f}{\too} & W'
 \end{array}
 $$
for some isomorphism $f:W\too W'$. The ${\mathbb C}^*$-action by
dilations on the fibers of $T^*M$
and $T^*M'$
induce ${\mathbb C}^*$-actions on $W$ and $W'$, and $f$ should be
${\mathbb C}^*$-equivariant (as $d\Phi$ is ${\mathbb C}^*$-equivariant).
Therefore $f: W_{r}\too W_{r}'$, and $f|_{W_r}$ is linear.

We have seen in Proposition \ref{prop:W-and-D}
that the Hitchin discriminant
${\mathcal D}$ is an intrinsically defined subset,
and therefore it is preserved by $f$. So $f$ preserves ${\mathcal C}={\mathcal D}\cap W_{r}$. This induces
an isomorphism of the corresponding dual varieties, and hence by
Proposition \ref{prop:4.2}, an isomorphism $\sigma: X \longrightarrow X'$ is obtained.
\end{proof}

\section{Automorphisms of the moduli space}

In this section we will compute the automorphism group of the moduli
spaces of stable bundles on $X$. We assume that $g\geq 4$.

Recall that there is a decomposition
 $$
  W=\bigoplus_{k=2}^r W_k=\bigoplus_{k=2}^r H^0(K_X^k)\, .
 $$

\begin{proposition} \label{prop:lll}
Fix a generic stable bundle $E\in M=M(r,\Lambda)$, and consider the map
 $$
 h_{r}: H^0(\End_0 E\otimes K_X) \too W_{r} \, ,
 $$
given as composition of the Hitchin map on
$H^0(\End_0 E\otimes K_X)=T^*_EM \subset T^*M$,
followed by the projection $W\longrightarrow W_{r}$.
Then
 $$
 H^0(\End_0 E\otimes K_X(-x_0))
 = \big\{ \psi \in H^0(\End_0 E\otimes K_X) \, | \, h_{r}(\psi+\phi)
 \in {\mathcal H}_{x_0}, \forall \phi\in h_{r}^{-1}({\mathcal H}_{x_0}) \big\}\, ,
 $$
where ${\mathcal H}_x=H^0(K_X^{r}(-x)) \subset W_{r}=H^0(K_X^r)$, for $x\in X$.
\end{proposition}

\begin{proof}
First, note that the sequence
 $$
 0\too H^0(\End_0 E\otimes K_X(-x_0)) \too H^0(\End_0 E\otimes K_X)
\too \End_0 E\otimes K_X|_{x_0} \too 0
 $$
is exact, since
$H^1(\End_0 E\otimes K_X(-x_0))=H^0(\End_0 E (x_0))^*= 0$,
for a general bundle (see Remark \ref{rem-v}).
So the map
 $$
 H^0(\End_0 E\otimes K_X) \too \End_0 E\otimes K_X|_{x_0}\, ,
 $$
given by $\phi \longmapsto \phi(x_0)$ is surjective.

Note that $h_{r}(\phi) =\det(\phi) \in W_{r}=H^0(K_X^{r})$. So
 $$
 h_{r}(\phi)\in {\mathcal H}_{x_0} \iff \det(\phi(x_0))=0\, .
 $$
The result now follows from this easy linear algebra fact:
Ff $V$ is a vector space and $A\in \End_0 V$ satisfies the
condition that $\det(A+C)=0$ for any
$C\in \End_0 V$ with $\det (C)=0$, then $A=0$.
\end{proof}

Proposition \ref{prop:lll} allows us to construct the vector bundle
 $$
 \SE \too X
 $$
whose fiber over $x\in X$ is $\SE_x = H^0(\End_0 E \otimes K_X(-x))$.
This is a
subbundle of the trivial vector bundle
 $$
 H^0(\End_0 E\otimes K_X) \otimes_{\mathbb C} {\mathcal O}_X \longrightarrow X\, ,
 $$
and there is an exact sequence
 \begin{equation}\label{eqn:28aug}
 0 \too \SE \too H^0(\End_0 E \otimes K_X) \otimes {\mathcal O}_X
 \stackrel{\pi}{\longrightarrow} \End_0 E\otimes K_X\too 0 \, .
 \end{equation}
(It is exact on the right by Remark \ref{rem-v}.)
So we recover the vector bundle
 $$
 \End_0 E\otimes K_X\too X\, .
 $$

\medskip

Let $x\in X$. We consider the nilpotent cone spaces
 $$
 {\mathcal N}_{E,x} =\{ A \in \End E\otimes K_X|_{x} \, \, | \, \, A^{r}=0\}
\subset \End_0 E\otimes K_X|_{x}\, .
 $$
Consider the map
 $$
 h_{k}: H^0(\End_0 E\otimes K_X) \too W_{k}
 $$
for $2\leq k\leq r$, given as composition of the Hitchin map on $H^0(\End_0  E\otimes K_X)=T^*_EM \subset T^*M$,
followed by the projection $W\longrightarrow W_{k}$.
 The map $h_{k}$ is defined by $(E,\phi)\longmapsto \tr (\wedge^{k}\phi)$. Therefore,
  \begin{equation}\label{eqn:nilp-x}
 \widetilde{\mathcal N}_{E,x} := \bigcap_{k=2}^r h_k^{-1} (H^0(K_X^k(-x))) \subset H^0(\End_0 E \otimes K_X)\,
 \end{equation}
is the preimage of the
nilpotent cone under the surjective map
$H^0(\End_0 E \otimes K_X) \too (\End_0 E \otimes K_X)|_x$.
Take its image to get ${\mathcal N}_{E,x}$.
Letting $x$ vary over $X$, we get the nilpotent cone bundle
 $$
 {\mathcal N}_E \too X
 $$
which sits as a sub-bundle
  $$
  {\mathcal N}_E \subset \End_0 E\otimes K_X .
  $$

\medskip

So the moduli space $M= M(r,\Lambda)$ allows us to recover the nilpotent cone bundle. Let us
see that we can also recover the flag bundle $Fl(E) \too X$ of complete
flags on the fibers of $E$.

\begin{lemma} \label{lem:flag}
Let $V$ be a vector space, and define the \textit{nilpotent cone} of $V$ as:
 $$
 {\mathcal N} = \{ A\in \End V \, | \, A^{r}=0\} = \{ A\in \End_0 V \, | \, A^{r}=0\},
 $$
where $r=\dim V$. Then ${\mathcal N}\subset \End V$ determines the flag variety $Fl(V)$
of complete flags in $V$.
\end{lemma}

\begin{proof}
This is a consequence of a theorem of Gerstenhaber \cite{Gerst}. Each complete flag
in $V$ determines a linear subspace $L\subset {\mathcal N} \subset \End V$ consisting
of nilpotent matrices respecting the flag. The dimension is
 $$
\dim L= \frac12 (r^2-r)= \frac12 \dim {\mathcal N}\, .
 $$
Conversely, any linear subspace $L\subset {\mathcal N}$ of
dimension $\frac12 (r^2-r)$ is determined by a unique flag in this way. So the flag
variety $Fl(V)$ is the space parametrizing these linear subspaces.
\end{proof}

Now we are ready to prove the main result of the paper.

\begin{theorem} \label{thm:autom}
 Let $M=M (r,\Lambda)$ be the moduli space of stable vector bundles. Let
$\Phi:M \too M$ be an automorphism.
Then $\Phi$ is an automorphism of the type
described in the introduction.
\end{theorem}

\begin{proof}
 As in the proof of Theorem \ref{thm:torelli}, the automorphism $\Phi$
yields an isomorphism $\sigma:X\too X$.
 Composing with an automorphism given by $\sigma^{-1}$,
we may assume that $\sigma={\rm Id}$.
 Take a \emph{general} bundle $E$,
 and let $E'$ be its image by $\Phi$. Then we have a diagram
 $$
  \begin{array}{ccc}
  T^*_EM & \stackrel{F=d\Phi}{\too} & T^*_{E'}M \\
  h \Big\downarrow \  & & h \Big\downarrow \ \\
  W & \stackrel{f}{\too} & W
  \end{array}
$$
where $f$ is an automorphism which commutes with
the ${\mathbb C}^*$-action.

\medskip

We will show now that $f={\rm Id}$.

First note that $\Pic M=\ZZ$ (\cite{Drezet-Narasimhan}). Therefore
$\Pic T^* M =\ZZ$. So the generic fiber of the Hitchin map has Picard group
$\ZZ$. Let $s\in W$ be a generic point, and let $s'=f(s)\in W$.
The isomorphism $d\Phi$ induces an isomorphism $$F:h^{-1}(s)\longrightarrow
h^{-1}(s')\, .$$ Since the compactification
$H^{-1}(s)$ of $h^{-1}(s)$ is an abelian variety, and $\codim (H^{-1}(s)- h^{-1}(s))\geq 2$,
we conclude that $F$ extends to an isomorphism $$F: H^{-1}(s) \longrightarrow
H^{-1}(s')$$ of polarized abelian varieties.

Let $X_s$ be the spectral curve (\ref{eqn:spectral-curve})
associated to $s\in W$. Recall that $P_s:=H^{-1}(s)$ is the Prym variety for the covering $X_s \longrightarrow X$.
Therefore there is an isomorphism of rational Hodge structures $H^1(X_s) \cong H^1(X) \oplus H^1(P_s)$.
Moreover, the natural polarization of $H^1(X_s)$ is of the form $a \, \Theta_X + b \, \Theta_s$,
where $a,b\in \QQ_{>0}$, $\Theta_X$ is the natural polarisation of $H^1(X)$, and $\Theta_s$ is
the unique polarisation of $H^1(P_s)$. We observe that this construction can be done over families. Therefore, $a$ and
$b$ are constants. Since $(H^1(P_s), \Theta_s) \cong (H^1(P_{s'}), \Theta_{s'})$, we conclude
that $H^1(X_s)\cong H^1(X_{s'})$ as polarized Hodge structures. By the usual Torelli
theorem, $\tilde{f}:X_s\stackrel{\cong}{\too} X_{s'}$. The isomorphism $H^1(X_s)\cong H^1(X_{s'})$ preserves the factor $H^1(X)$.
Therefore $\tilde{f}:X_s\to X_{s'}$
commutes with the natural projections $X_s \longrightarrow X$ and $X_{s'}\longrightarrow X$.

Recall that $X_s,X_{s'}\subset {\mathbb V}(K_X)$. The isomorphism
sends each point $p\in X_s$ to a point $\tilde{f}(p)\in X_{s'}$.
Both points are in the same fiber of ${\mathbb V}(K_X)$,
so, if $p$ is not in the zero section,
the quotient $\tilde{f}(p)/p$ is a well-defined complex number.
If we vary $p$, we obtain a rational function $\lambda(p)$ on $X_s$,
such that $\tilde{f}$ can be written as $\tilde{f}(p)=\lambda(p)p$.
This function has poles at the intersection of $X_s$ with the zero section of
${\mathbb V}(K_X) \to X$. So, it is of the form $\lambda=g/y$,
where $y\in H^0(X_s,\pi^*K_X)$
is the tautological coordinate $dx$ along the fibers
of the projection ${\mathbb V}(K_X)\too X$
and $g$ is a section of $H^0(X_s,\pi^*K_X)$.

Now
\begin{eqnarray*}
H^0(X_s, \pi^* K_X) &=& H^0(X, (\pi_* {\mathcal O}_{X_s}) \otimes K_X)=
H^0(X, \bigoplus_{j=0}^{r-1} K_X^{1-j}) = \\
 &=& H^0(X, {\mathcal O}_X) \oplus H^0(X, K_X) \, .
\end{eqnarray*}
In this decomposition, the tautological section $y$ is the identity
on the first summand. Therefore, we can write $g=g_1 y +g_2$,
where $g_1\in \CC$ and $g_2\in H^0(X,K_X)$.
Therefore, the isomorphism $\tilde{f}$ sends a point $p\in X_s$ with
coordinate $y$ to the point
$$
\lambda(x)y= \left(g_1 + \frac{g_2}{y}\right) y = g_1 y + g_2\, .
$$
Therefore, the effect of $g_1\in \CC$ is to produce a dilation with
constant factor, and the effect of $g_2$ is to produce a translation
along the fiber. Note that $g_2$ is defined as a section on the curve
$X$, so the translation is the same for all points of $X_s$ over the
same point of $X$.
As we are requiring $s_1(x)=0$ for the spectral curves, we must have
$g_2=0$. Therefore, the isomorphism $\tilde{f}$ is just a dilation by a
(nonzero) constant factor.

Summarizing, we have proved that for general $s$,
 $$
 f(s)= \lambda_s \, s \, ,
 $$
for some $\lambda_s \in \CC^*$.
This should then hold for all $s\in W$. Hence $\lambda_s$ is constant, and $f$ is multiplication
by a constant $\lambda$. After scaling by $\lambda^{-1}$, we get that $f={\rm Id}$.

\medskip

Now there is a vector bundle $\SE \too X$, whose fiber
over any $x\in X$ is the subspace $H^0(\End_0 E\otimes K_X(-x))\subset
T_E^*M$. Analogously, there
is a vector bundle $\SE'\too X$ whose fiber over any $x$ is
the subspace $H^0(\End_0 E'\otimes K_X(-x))\subset
T_{E'}^*M$. As $h \circ F=h$, Proposition \ref{prop:lll} implies that
$F: T_E^*M \too T_{E'}^*M$ gives an
isomorphism $\SE \too \SE'$.
Passing to the quotient bundle (\ref{eqn:28aug}), we have an isomorphism of vector bundles
 $$
 \End_0 E\otimes K_X \longrightarrow \End_0 E'\otimes K_X\, .
 $$
The map $F$ preserves the subspaces (\ref{eqn:nilp-x}), again because
$h\circ F=h$. Therefore there is an isomorphism
 $$
 \begin{array}{ccc}
 {\mathcal N}_E & {\too} & {\mathcal N}_{E'} \\
 \Big\downarrow & & \Big\downarrow \\
 X & = & X
 \end{array}
 $$
where ${\mathcal N}_E$ and ${\mathcal N}_{E'}$ are the corresponding nilpotent cone bundles.
By Lemma \ref{lem:flag}, we get an isomorphism
 $$
 \begin{array}{ccc}
 Fl(E) & {\too} & Fl({E'}) \\
 \Big\downarrow & & \Big\downarrow \\
 X & = & X
 \end{array}
 $$
of the corresponding flag variety bundles. Considering
the global vertical fields, we have a Lie algebra bundle isomorphism
 $$
 \begin{array}{ccc}
 \End_0 E & {\too} & \End_0 {E'} \\
 \Big\downarrow & & \Big\downarrow \\
 X & = & X
 \end{array}
 $$
Using Lemma \ref{endiso} below, it follows that $E'\cong E\otimes L$ or $E'\cong
E^\vee \otimes L$, for some line
bundle $L$. In the first case, it should be $L^r\cong \SO_X$.
In the second case, $L^r \cong \Lambda^2$ (and in particular $r|2d$).

As this holds for a generic $E$, it holds for all $E$.
\end{proof}

\begin{lemma}
\label{endiso}
Let $E$ and $E'$ be vector bundles of rank $r$ such
that $\End_0 E$ and $\End_0 E'$ are isomorphic
as Lie algebra bundles. Then there is a line
bundle $L$ such that, either $E'\cong E\otimes L$,
or $E' \cong E^\vee\otimes L$.
\end{lemma}

\begin{proof}

Giving a vector bundle $\End_0 E$ with
its Lie algebra structure is equivalent to giving
a principal ${\rm Aut}(\mathfrak{sl}_r)$-bundle $P_{{\rm Aut}(\mathfrak{sl}_r)}$ which admits
a reduction of structure group to a principal $\GL_r$-bundle $P_{\GL_r}$ (the
one corresponding to $E$).

We will study this reduction
in two steps:
$$
\GL_r \surj \PGL_r={\rm Inn}(\mathfrak{sl}_r)={\rm Aut}(\mathfrak{sl}_r)^0
 \inj {\rm Aut}(\mathfrak{sl}_r)\, ,
$$
where ${{\rm Aut}(\mathfrak{sl}_r)^0}$ is the connected component of the
identity, which coincides with the inner automorphisms
${\rm Inn}(\mathfrak{sl}_r)$.
There is
a short exact sequence of groups
$$
1 \too {\rm Inn}(\mathfrak{sl}_r) \too {\rm Aut}(\mathfrak{sl}_r) \too
{\rm Out}(\mathfrak{sl}_r)=\ZZ/2\ZZ \too 1\, .
$$
Hence, the space of reductions of structure group of
$P_{{\rm Aut}(\mathfrak{sl}_r)}$ to $\PGL_r$ correspond to
sections of the associated bundle
$P_{{\rm Aut}(\mathfrak{sl}_r)}({\rm Out}(\mathfrak{sl}_r))$.
Since ${\rm Out}(\mathfrak{sl}_r)=\ZZ/2\ZZ$, this
associated bundle is
a 2-to-1 cover of $X$, which is trivial
(since we know that there are reductions), and
therefore there are
two reductions; these two reductions are related
by an outer automorphism. Therefore,
if $\PP\to X$ is the bundle of projective spaces corresponding to
one reduction, the other one
is $\PP^\vee\to X$, the dual projective bundle.

Now consider the short exact sequence of groups
 $$
 1 \too \CC^* \too \GL_r \too \PGL_r \too 1\, .
 $$
The set of isomorphism classes of
reductions of a $\PGL_r$-bundle
to $\GL_r$ is a torsor for the group $H^1(X,\SO^*_X)$.
Let $E$ be a vector bundle corresponding
to a reduction of $\PP\longrightarrow X$, i.e.,
$\PP(E)\cong\PP$. Then the reductions
correspond to vector bundles
of the form $E\otimes L$ for any line bundle $L$.
On the other hand, $E^\vee$ is a reduction of
$\PP^\vee\longrightarrow
 X$. Therefore, if $E$ is the vector bundle
corresponding to a reduction of $P_{{\rm Aut}(\mathfrak{sl}_r)}$, then
all reductions are of the form either $E\otimes L$ or $E^\vee\otimes L$,
where $L$ is a line bundles.
\end{proof}

\begin{remark}
{\rm Let $\overline{M}(r,\Lambda)$ be the moduli space of semistable vector bundles
on $X$ or rank $r$ and determinant $\Lambda$. The automorphism group of
$\overline{M}(r,\Lambda)$ coincides with that of $M(r,\Lambda)$. To see this,
first note
that since $g\, \geq\, 3$, the smooth locus of $\overline{M}(r,\Lambda)$ coincides
with $M(r,\Lambda)$. Hence we have an injective homomorphism
$$
{\rm Aut}(\overline{M}(r,\Lambda))\, \longrightarrow \, {\rm Aut}(M(r,\Lambda)),
$$
obtained by restricting maps. To prove that the above homomorphism
is surjective, we note that for any line bundle $\zeta$ on
$\overline{M}(r,\Lambda)$, the restriction homomorphism
\begin{equation}\label{eqca1}
H^0(\overline{M}(r,\Lambda), \, \zeta) \, \longrightarrow\, H^0(M(r,\Lambda),
\, \zeta\vert_{M(r,\Lambda)})
\end{equation}
is an isomorphism. Also, from the fact that the Picard group of
$M(r,\Lambda)$ is $\mathbb Z$, it follows that the action
of ${\rm Aut}(M(r,\Lambda))$ on ${\rm Pic}(M(r,\Lambda))$ is trivial.
Now, take a very ample line bundle $\zeta$ on $\overline{M}(r,\Lambda)$.
Since the homomorphism in \eqref{eqca1} is an isomorphism, and any
automorphism $T$ of $M(r,\Lambda)$ fixes the restriction of $\zeta$,
we conclude that $T$ acts on $H^0(\overline{M}(r,\Lambda),\, \zeta)$. The
corresponding automorphism of ${\mathbb P}(H^0(M(r,\Lambda),\zeta))$
clearly fixes the image of $\overline{M}(r,\Lambda)$. Therefore, $T$ extends
to an automorphism of $\overline{M}(r,\Lambda)$.}
\end{remark}

\end{document}